\documentclass[12pt]{amsart}%
\usepackage{amsmath}
\usepackage{amsthm}
\usepackage{amscd}
\usepackage{amsfonts}
\usepackage{amsbsy}
\usepackage{epsfig}
\usepackage{amssymb,latexsym,psfrag,lscape}
\usepackage{vmargin}
\usepackage{amssymb}
\usepackage{graphicx}

\newtheorem {theorem} {Theorem}
\newtheorem {proposition} [theorem]{Proposition}
\newtheorem {corollary} [theorem]{Corollary}
\newtheorem {lemma}  [theorem]{Lemma}

\newtheorem {remark} [theorem]{Remark}

\begin{document}

\title[Bifurcation set for a disregarded Bogdanov-Takens unfolding]{Bifurcation set for a disregarded Bogdanov-Takens unfolding. Application to 3D cubic memristor oscillators.}
\author[A. Amador and E. Freire and E. Ponce]{Andr\'es Amador$^{1}$ and Emilio Freire$^{2}$ and Enrique Ponce$^{3}$}
\footnote{Facultad de Ingenier\'{\i}a y Ciencias, Departamento de Ciencias
Naturales y Matem\'aticas, Pontificia Universidad Javeriana-Cali, Cali, Colombia.\\
$^{2,3}$ Departamento de Matem\'{a}tica Aplicada, Escuela T\'ecnica Superior
de Ingenier\'{\i}a, Avda. de los Descubrimientos, 41092 Sevilla, Spain.\\
$^{1}$afamador@javerianacali.edu.co, $^{2}$efrem@us.es $^{3}$eponcem@us.es}

\maketitle

\begin{abstract}
We derive the bifurcation set for a not previously considered three-parametric
Bogdanov-Takens unfolding, showing that it is possible express its vector
field as two different perturbed cubic Hamiltonians. By using several
first-order Melnikov functions, we obtain for the first time analytical
approximations for the bifurcation curves corresponding to homoclinic and
heteroclinic connections, which along with the curves associated to local
bifurcations organize the parametric regions with different structures of
periodic orbits.

As an application of these results, we study a family of 3D memristor
oscillators, for which the characteristic function of the memristor is a cubic
polynomial. We show that these systems have an infinity number of invariant
manifolds, and by adding one parameter that stratifies the 3D dynamics of the
family, it is shown that the dynamics in each stratum is topologically
equivalent to a representant of the above unfolding. Also, based upon the
bifurcation set obtained, we show the existence of closed surfaces in the 3D
state space which are foliated by periodic orbits. Finally, we clarify some
misconceptions that arise from the numerical simulations of these systems,
emphasizing the important role played by the existence of invariant manifolds.

\end{abstract}

\keywords{Bifurcation set, Bogdanov-Takens, homoclinic orbit, heteroclinic connection, Melnikov function, Memristor oscillators}

\section{Introduction}

In planar systems, the existence of some local bifurcations may reveal the
presence of other bifurcations of global character
\cite{dumortier1987,dumortier1991} and the curves that determine these global
phenomena are difficult to determine. This is, for instance, the case
regarding the appearance of homoclinic or heteroclinic connections.

A homoclinic connection is an orbit of the system that joins a saddle
equilibrium point to itself, and generally creates or destroys periodic orbits
(see, for instance \cite{wiggins2003}). A heteroclinic connection joins two
different equilibrium points of a system and the existence of this connection
can determine changes in the basin of attraction of a positively invariant set.

Following \cite{freire2000}, the techniques to study homoclinic orbits in
planar vector fields were well developed during the 1920s in the works of
Dulac. The fundamental idea is that the recurrent behavior near a connecting
orbit should be studied in a fashion similar to that used in studying periodic
orbits via a Poincar\'{e} return map. But there are some additional
complications in the study of homoclinic orbits compared to that of periodic
orbits which significantly complicate the analysis.

Usually, the bifurcation curves of homoclinic and heteroclinic connections are
studied by numerical continuation techniques
\cite{freire1999,freire2000,matcont2008,matcont2012}. On the other hand, when
a planar system can be written as a perturbed Hamiltonian system, we can
calculate some Melnikov functions, introduced by Melnikov in \cite{melnikov63}%
, and under certain hypotheses, the zeros of the associated Melnikov function
determine the existence of periodic orbits, homoclinic loops or heteroclinic
connections, see for instance \cite{dangelmayr1987,guckenheimer1989,perko1994}.

As show later, we will resort to such Melnikov functions for setting
information on such global bifurcations curves in a two-parametric plane for a
specific family of differential systems.

The normal form of the Bogdanov-Takens (see \cite{wiggins2003}) bifurcation is
given by
\[
\dot{x}=y,\quad\dot{y}=\mu_{1}+\mu_{2}x+x^{2}\pm xy.
\]
Following the classification proposed in \cite{dumortier1991}, the deformation
of codimension three of the previous normal form is given by the unfolding
\begin{equation}%
\begin{split}
\dot{x}  &  =y,\\
\dot{y}  &  =\mu_{1}+\mu_{2}x+\alpha x^{3}+y(\mu_{3}+\mu_{4}x\pm x^{2}),
\end{split}
\label{unfolding}%
\end{equation}
The presence of this type of systems has been reported in different
applications, see \cite{freire2005,kuznetsov2016,kong2017}. On the study of
bifurcation phenomena in these systems many contributions have been made. In
\cite{freire1996b}, the authors studied the global bifurcation diagram of the
three-parameter family%
\begin{align*}
\dot{x}  &  =y,\\
\dot{y}  &  =\mu_{1}+\mu_{2}x-x^{3}+y(\mu_{3}-3x^{2}),
\end{align*}
and fixing $\mu_{3}>0$, they obtained analytical approximations to the
bifurcation curves of the homoclinic orbits, by using Melnikov functions.
Later on, the above work was quoted in \cite{khibnik1998}, where a numerical
analysis of the same model was performed. In
\cite{dumortier2001,dumortier2003b,dumortier2003}, it is considered the
system
\begin{equation}%
\begin{split}
\dot{x}  &  =y,\\
\dot{y}  &  =\mu_{1}+\mu_{2}x-x^{3}+y(\mu_{3}+\mu_{4}x-x^{2}),
\end{split}
\label{sys-aux}%
\end{equation}
and the authors showed that it can be written as a perturbed Hamiltonian
system, reporting the maximum number of limit cycles. Later, by taking the
parameter $\mu_{4}=0$ in \eqref{sys-aux}, the authors in
\cite{chen2015,chen2016,Hebai2017,Hebai2018} analyzed the system as a
Li\'{e}nard system, \ its local bifurcations were characterized, and a
numerical study of the global bifurcations was done.

While all the above references dealt with the focus case, in this work we
study the saddle case
\begin{equation}%
\begin{split}
\dot{x}  &  =y,\\
\dot{y}  &  =\mu_{1}+\mu_{2}x+x^{3}+y(\mu_{3}-3x^{2}),
\end{split}
\label{forma-canonica-1A}%
\end{equation}
which up to the best of our knowledge, seems to be a disregarded case with
rather interesting dynamic behavior.

In fact, our motivation comes from the analysis of certain 3D memristor
oscillators \cite{amador2017,ponce2017,chua2008}, where under specific
hypotheses on the memristor characteristics, such system appears in a natural
way after a dimensional reduction achieved thanks to the existence of a first integral.

The paper is organized in the following way. First, in section \ref{sec:2} we
review the information that can be gained by means of a local analysis of the
system. Our main results appear in section \ref{sec:3}, where we apply
Melnikov theory to approximate the homoclinic and heteroclinic curves on a
convenient parameter plane. We obtain the global bifurcation set for system
\eqref{forma-canonica-1A}, by splitting such a plane in regions with different
qualitative dynamical behavior. Next, in section \ref{sec:4}, We show how the
above analysis is useful for deriving all the possible responses of certain 3D
canonical memristor oscillators, when the flux-charge characteristics function
is a specific cubic polynomial, generalizing some results given in
\cite{amador2017}. As one of the possible dynamical behaviors, we focus or
attention in showing, thanks to the previous analysis, the existence of a
topological sphere in the 3D phase-space completely foliated by periodic
orbits. Thus, we confirm previous numerical results reported in
\cite{messias2010,Korneev2017}. The necessity of incorporating rigorous
techniques in the analysis of memristor oscillators is emphasized with the
material of section \ref{sec:5}, where following a similar procedure for the
dimensional reduction of section \ref{sec:4}, we can refute several recently
published studies that report the existence of an infinite number of hidden
attractors in a three-dimensional memristor-based autonomous Duffing
oscillator. Some technical results are relegated to the appendix.

\section{Local bifurcations}

\label{sec:2}

In this section, we study the local bifurcations that occur in system
\eqref{forma-canonica-1A}. First, we note that the system is invariant under
the transformation
\[
(x,y,\mu_{1},\mu_{2},\mu_{3})\rightarrow(-x,-y,-\mu_{1},\mu_{2},\mu_{3}).
\]
Therefore, it is sufficient to study the bifurcation diagram for $\mu_{1}>0.$
The equilibrium points of the system are of the form $(\overline{x}%
,\overline{y})=(\tilde{x},0)$, being $\tilde{x}$ a solution of the cubic
$\mu_{1}+\mu_{2}x+x^{3}=0,$ and its Jacobian matrix of is given by%
\begin{equation}
J(x,y)=%
\begin{pmatrix}
0 & 1\\
\mu_{2}+3x^{2}-6yx & \mu_{3}-3x^{2}%
\end{pmatrix}
. \label{jacobian}%
\end{equation}

\begin{remark}
Note that for $\mu_{3}\leq0$ the divergence of system
\eqref{forma-canonica-1A} does not change sign, thus from Bendixson's
criterion \cite{Kocak91}, the system does not have periodic solutions.
\end{remark}

First, we provide a technical result that provides a study of the number of
equilibria in system \eqref{forma-canonica-1A} and their topological nature.

\begin{lemma}
\label{equiCubi} Consider system \eqref{forma-canonica-1A}, the following
statements hold.

\begin{itemize}
\item[(a)] If $\mu_{2}\geq0$ or we have $\mu_{2}<0$ with $27\mu_{1}^{2}%
+4\mu_{2}^{3}>0,$ then the system has only one equilibrium point.

\item[(b)] If $\mu_{2}<0$ and $27\mu_{1}^{2}+4\mu_{2}^{3}=0$ we have two
equilibrium points.

\item[(c)] If $\mu_{2}<0$ and $27\mu_{1}^{2}+4\mu_{2}^{3}<0,$ then the system
has three equilibrium points $\mathbf{x}_{i}=(s_{i},0)$ with $i\in\{L,C,R\}$
such that%
\begin{equation}
s_{L}<-\left(  -\mu_{2}/3\right)  ^{1/2}<s_{C}<\left(  -\mu_{2}/3\right)
^{1/2}<s_{R}, \label{cota-roots}%
\end{equation}
and $s_{L}+s_{C}+s_{R}=0.$ Furthermore, $\mathbf{x}_{L}$ and $\mathbf{x}_{R}$
are saddles while $\mathbf{x}_{C}$ is an antisaddle (node or focus).
\end{itemize}
\end{lemma}

\begin{proof}
We study the roots of the polynomial $p(x)=\mu_{1}+\mu_{2}x+x^{3}.$ Since
$p^{\prime}(x)=\mu_{2}+3x^{2},$ if $\mu_{2}\geq0$ we obtain $p^{\prime}%
(x)\geq0$ and so the polynomial has only one root. \

In the rest of the proof we assume $\mu_{2}<0$. The derivative $p^{\prime}(x)$
vanishes at the points $x_{\pm}=\pm(-\mu_{2}/3)^{1/2}$ being a maximum and
minimum local respectively, also a direct computation gives $p(x_{\pm}%
)=\mu_{1}\mp2(-\mu_{2}/3)^{3/2}$. When $p(x_{-})<0$ or $p(x_{+})>0$ the graph
of $p(x)$ only crosses once the \thinspace$x$-axis and so these inequalities
provides the condition $27\mu_{1}^{2}+4\mu_{2}^{3}>0,$ and the statement (a)
follows. Assuming $p(x_{-})=0$ or $p(x_{+})=0,$ the statement (b) follows.

Finally, if $p(x_{+})<0<p(x_{-})$ then we have three roots as indicated in
\eqref{cota-roots}. Moreover, using the relation between roots and
coefficients of polynomials, \ we get
\[
\mu_{1}=-s_{L}s_{C}s_{R},\quad\mu_{2}=s_{C}s_{L}+s_{C}s_{R}+s_{L}s_{R},\quad
s_{L}+s_{C}+s_{R}=0.
\]
For the Jacobian matrix given in \eqref{jacobian}, we get $J(s_{i}%
,0)=-(\mu_{2}+3s_{i}^{2})=-p^{\prime}(s_{i}).$ As we know that $p^{\prime
}(s_{L})>0$, $p^{\prime}(s_{C})<0$ and $p^{\prime}(s_{R})>0,$ the conclusion
follows and the proof is complete.
\end{proof}

In the next result, we give a characterization of the local bifurcations of
system \eqref{forma-canonica-1A} on the parametric plane $(\mu_{2},\mu_{1}),$
assuming a fixed value of the parameter $\mu_{3}.$ Note that wee have put the
$\mu_{1}$ axis in the plane $(\mu_{2},\mu_{1})$ vertically, being the $\mu
_{2}$ axis the horizontal one.

\begin{proposition}
The following statements hold for system \eqref{forma-canonica-1A}.

\begin{itemize}
\item[(a)] Given $\mu_{3}\in\mathbb{R}$ the parameter values in the set
\begin{equation}
\varphi_{sn}=\{(\mu_{2},\mu_{1}):27\mu_{1}^{2}+4\mu_{2}^{3}=0\},
\label{curveSN}%
\end{equation}
correspond with saddle-node bifurcation points of equilibria. In particular,
the system has a cusp bifurcation of equilibria at $\mu_{2}=\mu_{1}=0.$

\item[(b)] Given $\mu_{3}>0,$ the parameter values in the set%
\begin{equation}
\varphi_{H}=\{(\mu_{2},\mu_{1}):\mu_{1}=\pm\left(  \mu_{3}/3\right)  ^{3/2}%
\mp\left(  \mu_{3}/3\right)  ^{1/2}\mu_{2},\quad\text{ }\mu_{2}<-\mu_{3}\},
\label{curveH}%
\end{equation}
represent Andronov-Hopf bifurcation points of codimension one for the central
equilibrium point $\mathbf{x}_{C},$ see Lemma \ref{equiCubi}(c).

\item[(c)] The set defined in \eqref{curveH} determines a symmetric pair of
straight half lines emanating from two points corresponding to Bogdanov-Takens
bifurcation points, namely%
\begin{equation}
BT_{\pm}\equiv\left(  -\mu_{3},\pm2\left(  \mu_{3}/3\right)  ^{3/2}\right)  .
\label{BT-points}%
\end{equation}

\end{itemize}
\end{proposition}

\begin{proof}
Statement (a) is a direct consequence of the equations%
\[
x^{3}+\mu_{2}x+\mu_{1}=0,\quad\mu_{2}+3x^{2}=0,
\]
to be fulfilled for any non-hyperbolic equilibrium $(x,0)$ at a saddle-node bifurcation.

Let $\left(  \tilde{x},0\right)  $ an equilibrium point of system
\eqref{forma-canonica-1A}. Considering the Jacobian matrix $J$ given in
\eqref{jacobian}, then $J(\tilde{x},0)$ \ has two purely imaginary eigenvalues
when taking $\mu_{3}>0,$ the value $\tilde{x}$ satisfies $\tilde{x}=\pm
\sqrt{\mu_{3}/3}$ with $\mu_{2}<-\mu_{3}<0,$ because then $\mu_{2}+3\tilde
{x}^{2}<0.$ The last inequality is fulfilled only for the equilibrium point
$\mathbf{x}_{C},$ see Lemma \ref{equiCubi}(c). Since $\left(  \tilde
{x},0\right)  $ is an equilibrium point we have%
\[
\mu_{1}+\mu_{2}\left(  \pm\sqrt{\mu_{3}/3}\right)  +\left(  \pm\sqrt{\mu
_{3}/3}\right)  ^{3}=0,
\]
and statement (b) follows. To show statement (c) is sufficient to consider the
equations $\operatorname{trace}\left(  J(\tilde{x},0)\right)  =\det
(J(\tilde{x},0))=0.$
\end{proof}

\section{Global bifurcations}

\label{sec:3}

In this section we will complete the bifurcation analysis of system
\eqref{forma-canonica-1A}. We will write the system as a perturbed Hamiltonian
to which the Melnikov theory can be applied. This can be done in different
ways, as indicated in the next result. The possibility of resorting to one of
the two next reparametrization forms will be helpful later.

\begin{proposition}
System \eqref{forma-canonica-1A} can be written as two different perturbed
Hamiltonian systems, as follows.

\begin{itemize}
\item[(a)] Taking
\begin{equation}
\mu_{1}=\varepsilon^{4}\nu_{1},\quad\mu_{2}=-\varepsilon^{2}\nu_{2},\quad
\mu_{3}=\varepsilon^{2}\nu_{3}, \label{para-perturbedA}%
\end{equation}
the system can be rewritten as
\begin{equation}%
\begin{split}
\dot{x}  &  =y,\\
\dot{y}  &  =-\nu_{2}x+x^{3}+\varepsilon\left(  \nu_{1}+\nu_{3}y-3x^{2}%
y\right)  ,
\end{split}
\label{perturbed-systemA}%
\end{equation}
which for $\varepsilon=0$ corresponds to the Hamiltonian
\begin{equation}
H_{1}(x,y)=\frac{y^{2}}{2}+\nu_{2}\frac{x^{2}}{2}-\frac{x^{4}}{4}.
\label{HamiltonianA}%
\end{equation}

\item[(b)] Taking%
\begin{equation}
\mu_{1}=\varepsilon^{3}\nu_{1}\quad,\mu_{2}=-\varepsilon^{2}\nu_{2},\quad
\mu_{3}=\varepsilon^{2}\nu_{3}, \label{para-perturbed}%
\end{equation}
the system can be rewritten as
\begin{equation}%
\begin{split}
\dot{x}  &  =y,\\
\dot{y}  &  =\nu_{1}-\nu_{2}x+x^{3}+\varepsilon(\nu_{3}y-3x^{2}y),
\end{split}
\label{perturbed-system}%
\end{equation}

which for $\varepsilon=0$ corresponds to the Hamiltonian%
\begin{equation}
H_{2}(x,y)=\frac{y^{2}}{2}-\nu_{1}x+\nu_{2}\frac{x^{2}}{2}-\frac{x^{4}}{4}.
\label{Hamiltonian}%
\end{equation}

\end{itemize}
\end{proposition}

\begin{proof}
The blow-up transformation $\ x_{1}=(1/\varepsilon)x,$ $y_{1}=(1/\varepsilon
^{2})y,$ and $\tilde{t}=\varepsilon t$, allows to rewrite system
\eqref{forma-canonica-1A} as
\[
x_{1}^{\prime}=y_{1},\quad y_{1}^{\prime}=x_{1}^{3}+\frac{\mu_{2}}%
{\varepsilon^{2}}x_{1}+\frac{\mu_{1}}{\varepsilon^{3}}+\frac{\mu_{3}%
}{\varepsilon}y_{1}-3\varepsilon x_{1}^{2}y_{1},
\]
where the prime denotes derivatives with respect to the new time $\tilde{t}.$
Now, using \eqref{para-perturbedA} and \eqref{para-perturbed}, after some
elementary algebra we obtain systems \eqref{perturbed-systemA} and
\eqref{perturbed-system}, respectively.
\end{proof}

The phase portrait for the unperturbed Hamiltonian systems
\eqref{perturbed-systemA} and \eqref{Hamiltonian} are shown in Figure
\ref{Hamiltonian-fig}. Note that when $\nu_{1}=0$ we obtain $H_{1}%
(x,y)=H_{2}(x,y),$ and so in that case it is sufficient to study the
properties of the Hamiltonian $H_{1}.$

Now, we will consider the heteroclinic connections of unperturbed Hamiltonian
system \eqref{HamiltonianA}. The Hamiltonian has a pair of heteroclinic
connections $\Gamma_{\pm}(t)=(x\left(  t\right)  ,\pm y\left(  t\right)  ),$
parameterized by
\begin{equation}%
\begin{split}
x\left(  t\right)   &  =\sqrt{\nu_{2}}\tanh\left(  \sqrt{\nu_{2}/2}t\right)
,\\
y\left(  t\right)   &  =\frac{\nu_{2}}{\sqrt{2}}\operatorname{sech}^{2}\left(
\sqrt{\nu_{2}/2}t\right)  ,
\end{split}
\label{param-hetero}%
\end{equation}
where $-\infty<t<\infty$ and $\nu_{2}>0.$ In the next result, we compute the
Melnikov function along the heteroclinic connection $\Gamma_{+}$ for the
unperturbed Hamiltonian system \eqref{perturbed-systemA}, and by using
\eqref{para-perturbedA}, we obtain the approximate bifurcation curves for
heteroclinic connections for system \eqref{forma-canonica-1A}.

\begin{figure}[pth]
\begin{center}
\includegraphics[width=15cm]{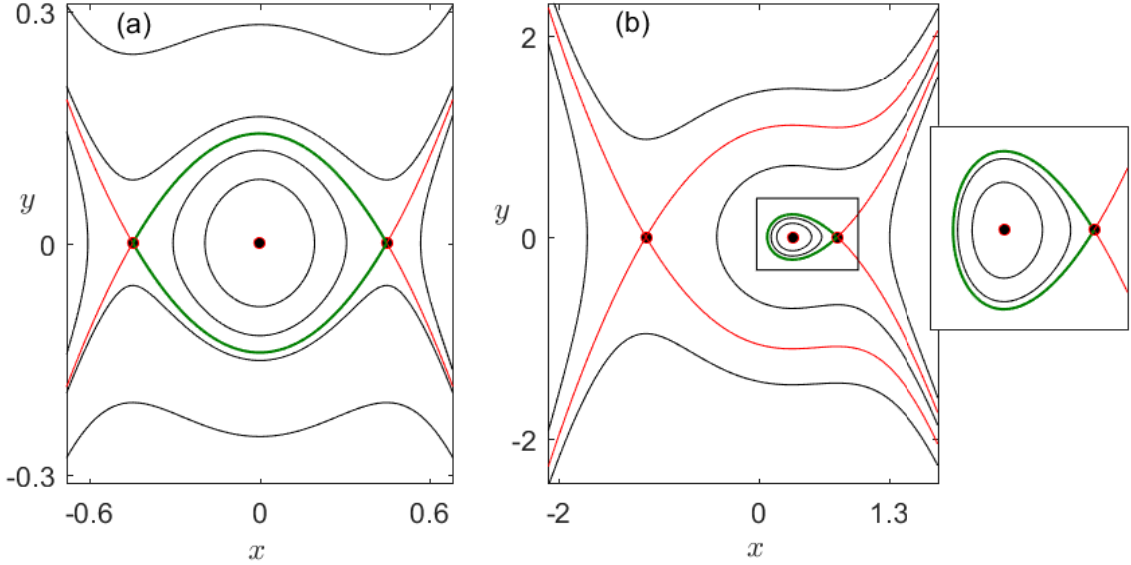}
\end{center}
\caption{(a) Phase portrait of unperturbed Hamiltonian system
\eqref{perturbed-systemA} with $\nu_{2}=0.2$. We show in green the two
heteroclinic orbits, while the non-closing stable and unstable manifolds of
the saddle points are shown in red. (b) Phase portrait of unperturbed
Hamiltonian system \eqref{perturbed-system} with $\nu_{1}=0.3$ and $\nu_{2}%
=1$. We draw in green the homoclinic orbit, the stable and unstable manifolds
of the saddle points are shown in red. }%
\label{Hamiltonian-fig}%
\end{figure}

\begin{proposition}
If we consider perturbed Hamiltonian system \eqref{perturbed-systemA} and
$\overline{\nu}=(\nu_{1},\nu_{2},\nu_{3})$ with $\nu_{2}>0$ and $\nu_{2}%
^{2}/4<\nu_{2}^{3}/27$ (see Lemma \ref{equiCubi}(c)) then the Melnikov
function along of the heteroclinic connection $\Gamma_{\pm}$ is given by%
\begin{equation}
M_{ht}(\overline{\nu})=\frac{2}{15}\sqrt{\nu_{2}}\left(  15\nu_{1}+5\sqrt
{2}\nu_{2}\nu_{3}-3\sqrt{2}\nu_{2}^{2}\right)  . \label{melnikovF1}%
\end{equation}

\end{proposition}

\begin{proof}
The system can be written as
\[
(\dot{x},\dot{y})^{T}=f(x,y)+\varepsilon g(x,y),
\]
where $f(x,y)=(y,-\nu_{2}x+x^{3})^{T}$ and $g(x,y)=(0,\nu_{1}+\nu_{3}%
y-3x^{2}y)^{T}.$ Thus, we have $f\wedge g=y\left(  \nu_{1}+\nu_{3}%
y-3x^{2}y\right)  $. Accordingly, the Melnikov function is defined by
\begin{align*}
M_{ht}(\overline{\nu})  &  =\int_{-\infty}^{\infty}f(x(t),\pm y(t))\wedge
g(x(t),\pm y(t))dt=\\
&  =\int_{-\infty}^{\infty}\pm y(t)\left[  \nu_{1}\pm(\nu_{3}-3x^{2}%
(t))y(t)\right]  dt,
\end{align*}
where $x(t)$ and $y(t)$ are defined as in \eqref{param-hetero}. After a direct
computation we obtain \eqref{melnikovF1}.
\end{proof}

By using the Melnikov theory and by fixing one parameter of system
\eqref{forma-canonica-1A}, we can give an approximation of the heteroclinic
connection curves in the remaining parameters plane.

\begin{proposition}
Consider system \eqref{forma-canonica-1A} with $\mu_{3}>0$ sufficiently small
and the parametric plane $(\mu_{2},\mu_{1})$. Then the system has a unique
hyperbolic heteroclinic connection in a neighborhood of the curve
\begin{equation}
\varphi_{ht}=\{(\mu_{2},\mu_{1})\in\mathbb{R}^{2}:\mu_{1}=\pm\frac{\sqrt{2}%
}{15}\mu_{2}\left(  3\mu_{2}+5\mu_{3}\right)  ,\quad\mu_{2}\neq-5\mu_{3}/3\}.
\label{curve-ht}%
\end{equation}

\end{proposition}

\begin{proof}
Fixing $\nu_{3}=1$ in the Melnikov function given in \eqref{melnikovF1}, and
imposing the condition $M_{ht}(\nu_{1},\nu_{2})=0,$ we obtain
\[
\nu_{1}=\frac{\sqrt{2}}{15}\nu_{2}\left(  3\nu_{2}-5\right)  .
\]
From \eqref{para-perturbedA} we get $\varepsilon=\sqrt{\mu_{3}}$, $\mu_{1}%
=\mu_{3}^{2}\nu_{1},$and $\mu_{2}=-\mu_{3}\nu_{2},$ so that\
\[
\mu_{1}=\mu_{3}^{2}\nu_{1}=-\mu_{3}^{2}\frac{\sqrt{2}}{15}\frac{\mu_{2}}%
{\mu_{3}}\left(  3\left(  -\frac{\mu_{2}}{\mu_{3}}\right)  -5\right)  ,
\]
and the conclusion follows.
\end{proof}

When $\mu_{1}=0$ and $\mu_{3}>0$ on the parameter plane $(\mu_{2},\mu_{1}),$
we obtain the point of double heteroclinic connections
\begin{equation}
DHT\equiv\left(  -5\mu_{3}/3,0\right)  . \label{dhtpoint}%
\end{equation}

We recall that Schecter's points are co-dimension two points defined by the
intersection of a saddle-node curve and a homoclinic or heteroclinic curve,
for more details see \cite{schecter}. Taking the intersection points of the
saddle-node bifurcation curve and the heteroclinic curves given in
\eqref{curveSN} and \eqref{curve-ht} respectively, we obtain a first-order
approximation of Schecter's points of the system. Since the system is
symmetric with respect to the parameter $\mu_{1},$ the system has four
Schecter's points (see Figure \ref{bifurcation-sets}), these points are
\begin{equation}%
\begin{split}
S_{1}^{\pm}  &  \equiv\rho_{1}\left(  (5/27),\mp(5\sqrt{10}/729)\left(
\sqrt{18\mu_{3}+5}+\sqrt{5}\right)  \right)  ,\\
S_{2}^{\pm}  &  \equiv\rho_{2}\left(  (5/27),\pm(5\sqrt{10}/729)\left(
\sqrt{18\mu_{3}+5}-\sqrt{5}\right)  \right)  ,
\end{split}
\label{schecter}%
\end{equation}
where
\[
\rho_{1}=\left(  9\mu_{3}+5-\sqrt{5}\sqrt{18\mu_{3}+5}\right)  ,\quad\rho
_{2}=\left(  \sqrt{5}\sqrt{18\mu_{3}+5}-9\mu_{3}-5\right)  .
\]

Now, by using the homoclinic connection of Hamiltonian system
\eqref{Hamiltonian}, we compute the associated Melnikov function for system
\eqref{forma-canonica-1A} when $\nu_{3}=1$.

\begin{proposition}
If we consider system \eqref{perturbed-system} and $\overline{\nu}=(\nu
_{1},\nu_{2},\nu_{3})$ with $\nu_{1}>0$, $\nu_{2}>0$ and $\nu_{3}=1$, then the
Melnikov function associated to the homoclinic orbit with connection point
$(0,s_{R})$, it is given by%
\begin{equation}
M(\overline{\nu})=\sqrt{2}\frac{\cosh^{2}(\theta)}{\cosh^{2}(\theta)+2}\left(
F_{1}(\theta)+\nu_{2}F_{2}(\theta)\right)  , \label{melniF2B}%
\end{equation}
where
\begin{equation}%
\begin{split}
F_{1}(\theta)=  &  720\theta-320\sinh\theta+240\theta\cosh^{3}\theta
-320\cosh^{2}\theta\sinh\theta-\\
&  -80\cosh^{4}\theta\sinh\theta+480\theta\cosh\theta,\\
F_{2}(\theta)=  &  1440\theta\cosh\theta-768\sinh\theta-\cosh^{3}%
\theta-1344\cosh^{2}\theta\sinh\theta-\\
&  -48\cosh^{4}\theta\sinh\theta
\end{split}
\label{melniF2B-2}%
\end{equation}

and $0<\theta<\infty$, with
\[
\cosh\theta=\frac{2s}{\omega},\quad\omega^{2}=2(\nu_{2}-s_{R}^{2})>0,\quad
\nu_{1}=\nu_{2}s_{R}-s_{R}^{3},
\]
being $s_{R}$ the biggest positive root of the equation $\nu_{1}-\nu
_{2}x+x^{3}=0$, see Figure \ref{Homoclina}.
\end{proposition}

\begin{proof}
We consider the unperturbed Hamiltonian system given in
\eqref{perturbed-system} with $\nu_{2}>0$. From Lemma \ref{equiCubi}(c), the
system has $3$ equilibrium points $\mathbf{x}_{i}=(s_{i},0)$, where
$\mathbf{x}_{L}$ and $\mathbf{x}_{R}$ are saddle points and $\mathbf{x}_{C}$
is a focus or node and
\[
s_{L}<s_{C}<s_{R},\quad s_{L}+s_{C}+s_{R}=0,\quad s_{L}s_{C}s_{R}=-\nu_{1}.
\]
We study only the case $\nu_{1}>0$, for the case $\nu_{1}<0$ is
analogous.\newline System \eqref{perturbed-system} can written as%
\[
(\dot{x},\dot{y})^{T}=f(x,y)+\varepsilon g(x,y).
\]
Now, assuming \ $\nu_{1}>0,$ by Green's Theorem, the homoclinic Melnikov
function of the system can rewritten as
\[
M_{h}(\overline{\nu})=\int\int_{D(\nu_{1},\nu_{2})}\left(  \frac{-\partial
g(x,y)}{\partial y}\right)  dA,
\]
where $D$ is the region bounded by the homoclinic orbit which joins the
equilibrium point $\left(  s_{R},0\right)  $ to itself. By fixing $\nu_{3}=1$
(that is $\mu_{3}>0$), and taking $\ p(x)=\nu_{1}-\nu_{2}x+x^{3},$ we get
$p(s_{R})=\nu_{1}-\nu_{2}s_{R}+s_{R}^{3}=0,$ that is
\begin{equation}
\nu_{1}=s_{R}(\nu_{2}-s_{R}^{2}), \label{auxilar-cubicA}%
\end{equation}
and so $\nu_{2}-s_{R}^{2}>0.$ Taking the auxiliary function
\[
q(x)=\int_{0}^{x}p(x)dx=\nu_{1}x-\nu_{2}\frac{x^{2}}{2}+\frac{x^{4}}{4},
\]
and using \eqref{Hamiltonian}, the homoclinic loop is given by the points
$\left(  x,y_{s}^{\pm}(x)\right)  $ where $\overline{x}\leq x\leq s_{R}$,
\[
y_{s}^{\pm}(x)=\pm\sqrt{2}\sqrt{q(x)-q(s_{R})},
\]
and $y_{s}^{\pm}(\overline{x})=y_{s}^{\pm}(s_{R})=0,$ see Figure
\ref{Homoclina}. Now, the Melnikov function is thanks to the symmetry of the
loop
\begin{align*}
M_{h}(\overline{\nu})  &  =2\int_{\overline{x}}^{s_{R}}(3x^{2}-1)dx\int
_{0}^{y_{s}^{+}(x)}dy=\\
&  =\sqrt{2}\int_{\overline{x}}^{s}(3x^{2}-1)(s_{R}-x)\sqrt{(x+s_{R}%
)^{2}-2(\nu_{2}-s_{R}^{2})}dx=\\
&  =\sqrt{2}\int_{\overline{x}}^{s}(3x^{2}-1)(s_{R}-x)\sqrt{(x+s_{R}%
)^{2}-\omega^{2}}dx,
\end{align*}
where from \eqref{auxilar-cubicA} $\omega^{2}=2(\nu_{2}-s_{R}^{2}),$ and we
have used that
\[
q(x)-q(s_{R})=\frac{1}{4}(x-s_{R})^{2}\left[  (x+s_{R})^{2}-\omega^{2}\right]
.
\]
Taking the change of variable
\[
x+s_{R}=\omega\cosh\theta,
\]
and noting that $q(s_{R})-q(\overline{x})=0$ we see that $\overline{x}%
+s_{R}=\omega$, which corresponds to $\theta=0,$ while for $x=s_{R}$ the
corresponding values of $\theta=\theta_{s_{R}}$ satisfy $\cosh\theta_{s_{R}%
}=2s_{R}/\omega,$ or also $\omega^{2}\cosh^{2}\theta_{s_{R}}=4s_{R}^{2},$ that
is $(s_{R}^{2}+\nu_{2})\cosh^{2}\theta_{s_{R}}=2s_{R}^{2},$ and so we get
\[
s_{R}^{2}=\frac{\cosh^{2}\theta_{s_{R}}}{2+\cosh^{2}\theta_{s_{R}}}\nu_{2}.
\]
Now we arrived to
\[
M_{h}(\overline{\nu})=\sqrt{2}\omega^{2}\int_{0}^{\theta_{s_{R}}}%
(1-3(\omega\cosh\theta-s_{R})^{2})(2s_{R}-\omega\cosh\theta)\sinh^{2}\theta
d\theta,
\]
and after some computations, we obtain \eqref{melniF2B} and
\eqref{melniF2B-2}, where $\theta_{s_{R}}$ has been simplified to $\theta$.
\end{proof}

\begin{figure}[pth]
\begin{center}
\includegraphics[width=11cm]{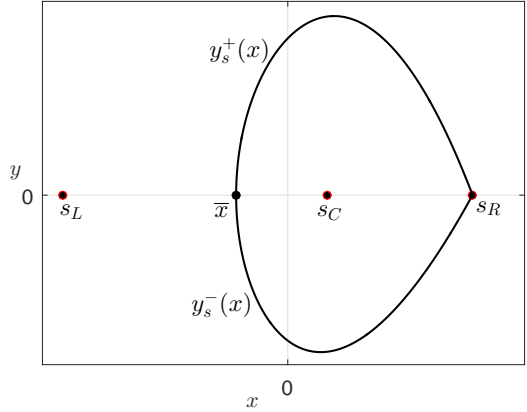}
\end{center}
\caption{ Homoclinic orbit which joins the saddle equilibrium point
$(s_{R},0)$ to itself.}%
\label{Homoclina}%
\end{figure}

As a direct consequence of the above result, we give an analytical
approximation of the bifurcation curves for homoclinic connections of system \eqref{forma-canonica-1A}.

\begin{proposition}
Consider system \eqref{forma-canonica-1A} with $\mu_{3}>0$ sufficiently small
and the parametric plane $(\mu_{2},\mu_{1}).$ Then the system has a unique
homoclinic orbit in a neighborhood of the curve
\begin{equation}
\varphi_{h}=\{(\mu_{2},\mu_{1})\in\mathbb{R}^{2}:\mu_{2}=-\mu_{3}\nu
_{2}(\theta),\quad\mu_{1}=\pm\mu_{3}^{3/2}\nu_{1}(\theta),\quad0<\theta
<\infty\}, \label{curve-h}%
\end{equation}
where
\begin{equation}%
\begin{split}
\nu_{2}\left(  \theta\right)   &  =\frac{10(\cosh2\theta+5)(9\sinh\theta
+\sinh3\theta-12\theta\cosh\theta)}{3(370\sinh\theta+115\sinh3\theta
+\sinh5\theta-60\theta(11\cosh\theta+\cosh3\theta))},\\
\nu_{1}(\theta)  &  =\nu_{2}\left(  \theta\right)  s-s^{3},\quad s^{2}%
=\frac{\cosh^{2}\theta}{2+\cosh^{2}\theta}\nu_{2}(\theta).
\end{split}
\label{parametric-curve}%
\end{equation}
Moreover, for the points $(\mu_{2}(\theta),\mu_{1}(\theta))$ at the curve
$\varphi_{h}$ we have.%
\[
\lim_{\theta\rightarrow0^{+}}(\mu_{2},\mu_{1})=(-\mu_{3},\pm2/3\sqrt{\mu
_{3}^{3}/3}),\quad\lim_{\theta\rightarrow\infty}(\mu_{2},\mu_{1})=\left(
-5\mu_{3}/3,0\right)  .
\]

\end{proposition}

\begin{proof}
The Melnikov function given in \eqref{melniF2B}-\eqref{melniF2B-2} vanishes at
the points $(\nu_{1}(\theta),\nu_{2}(\theta))$ defined in
\eqref{parametric-curve}. Taking $\nu_{3}=1$ in \eqref{para-perturbed} we
obtain $\mu_{1}=\mu_{3}^{3/2}\nu_{1}$ and $\mu_{2}=-\mu_{3}\nu_{2},$ and after
some computations the conclusion follows.
\end{proof}

\begin{remark}
Note that from the previous result we obtain the two points
\[
\lim_{\theta\rightarrow0^{+}}(\mu_{2},\mu_{1})\equiv BT,\quad\lim
_{\theta\rightarrow\infty}(\mu_{2},\mu_{1})\equiv DHT,
\]
where the points $BT$ and $DHT$ are given in \eqref{BT-points}\ and
\eqref{dhtpoint} respectively.
\end{remark}

In Figure \ref{bifurcation-sets}, the complete bifurcation set of system
\eqref{forma-canonica-1A} is shown. Figures \ref{limit-cycle} and
\ref{no-limit-cycle} give the different phase portrait in the labeled
parameter regions, where in these figures we show the different configurations
of the phase portrait of the system. In Figure \ref{limit-cycle}, since the
homoclinic Melnikov function is positive, we can guarantee that there is no
change on the relative position of the stable and unstable manifolds of each
saddle points.

\begin{figure}[pth]
\begin{center}
\includegraphics[width=15cm]{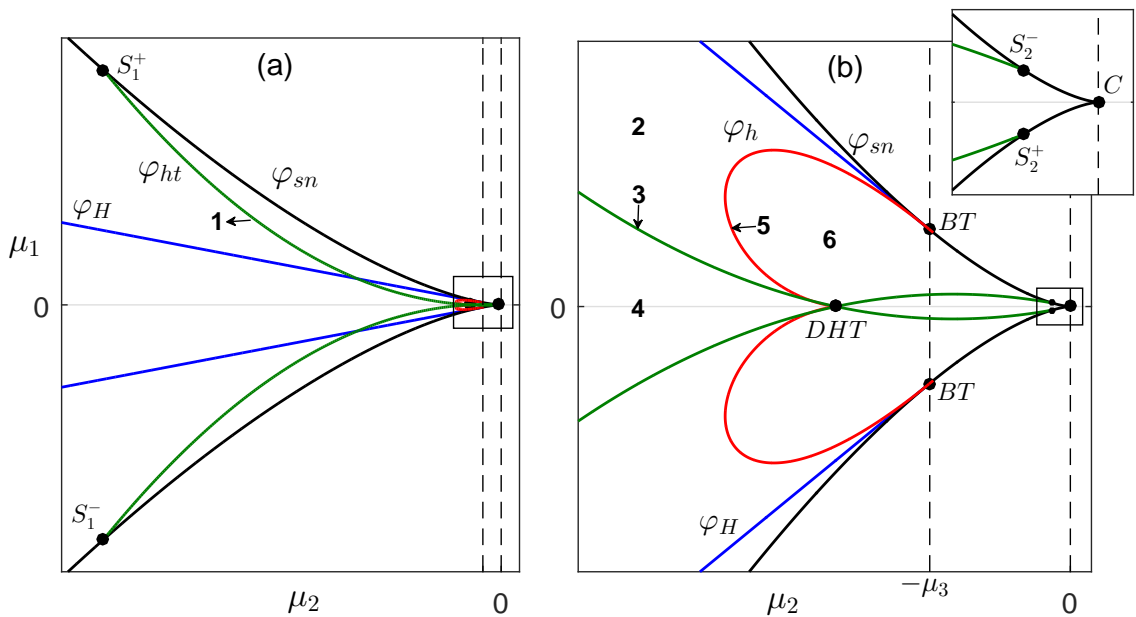}
\end{center}
\caption{The bifurcation diagram of system \eqref{forma-canonica-1A}, taking
$\mu_{3}>0$ sufficiently small.}%
\label{bifurcation-sets}%
\end{figure}

\begin{figure}[pth]
\begin{center}
\includegraphics[width=14cm]{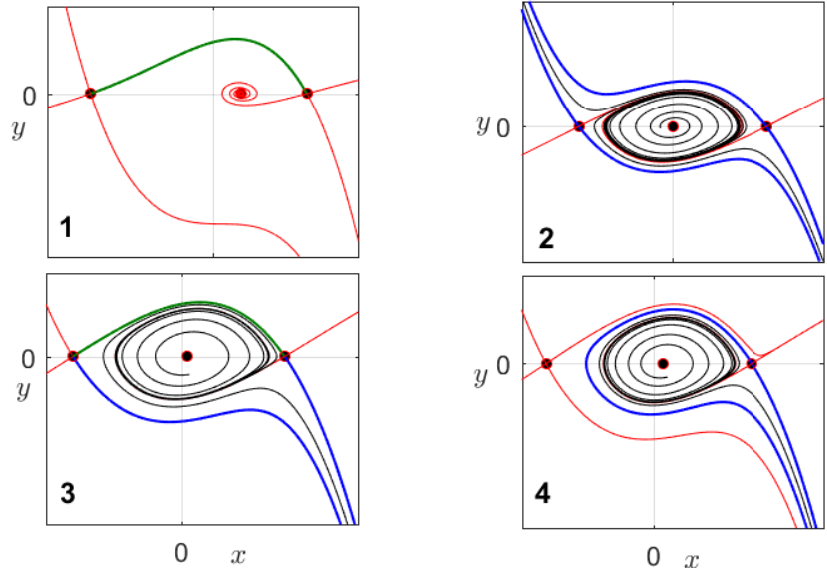}
\end{center}
\caption{Phase portrait of system \eqref{forma-canonica-1A} in the parameter
regions labeled with \textbf{1},\textbf{2}, \textbf{3} and \textbf{4} in
Figure \ref{bifurcation-sets}. The thick lines are the boundary of the basin
of attraction of a limit cycle, such boundary is formed by some stable and
unstable manifolds of the saddle points. The green lines are the heteroclinic
connections. }%
\label{limit-cycle}%
\end{figure}

\begin{figure}[pth]
\begin{center}
\includegraphics[width=12cm]{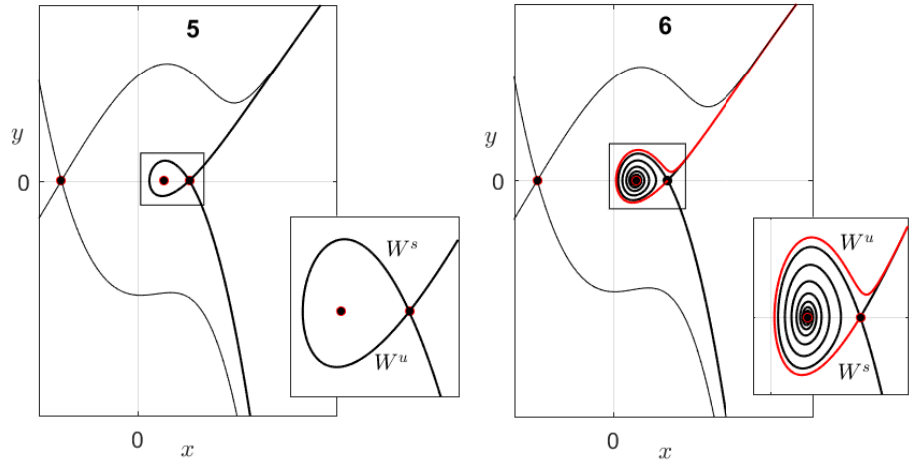}
\end{center}
\caption{Phase portrait of system \eqref{forma-canonica-1A} in the parameter
regions labeled with \textbf{5} and \textbf{6} in Figure
\ref{bifurcation-sets}. The stable and unstable manifolds of the saddle
points.}%
\label{no-limit-cycle}%
\end{figure}

\begin{remark}
\label{existence}Note that considering the function $\nu_{2}$ defined in
\eqref{parametric-curve}, and after some algebra, we obtain that finding the
minimum of the function $\nu_{2}$ is equivalent to finding the zeros of the
function
\begin{equation}
h_{1}\left(  x\right)  =2x\left(  26\cosh2x+\cosh4x+33\right)  -5\left(
10\sinh2x+\sinh4x\right)  , \label{fun-h}%
\end{equation}
where $h_{1}\left(  0\right)  =0,$ $h_{1}\left(  1\right)  <0$ and
$h_{1}\left(  2\right)  >0,$ see Figure \ref{estudio-homoclina}(c). Thus at
$\theta^{\ast}\approx1.8630981$ the function $\nu_{2}$ has a minimum given by
$\nu_{2}(\theta^{\ast})\approx2.454887$ (see Figure \ref{estudio-homoclina}%
(b)), so that%
\[
-\frac{5}{2}\mu_{3}<-\nu_{2}(\theta^{\ast})\mu_{3}<-\frac{5}{3}\mu_{3}.
\]
Now, if we consider system \eqref{forma-canonica-1A} with $\mu_{3}>0$
sufficiently small, $\mu_{2}<-(5/2)\mu_{3}$ and $\mu_{1}$ such that
\[
|\mu_{1}|<\left(  \frac{\mu_{3}}{3}\right)  ^{3/2}-\left(  \frac{\mu_{3}}%
{3}\right)  ^{1/2}\mu_{2},
\]
then the system has a stable limit cycle, see Figure \ref{estudio-homoclina}%
(a). This assertion is a direct consequence of \eqref{curveH} and
Poincar\'{e}-Bendixson Theorem (see for instance \cite{wiggins2003}), since
the sign of the Melnikov function guarantees the existence of a compact
positive invariant set with only one unstable equilibrium point in its interior.
\end{remark}

\begin{figure}[pth]
\begin{center}
\includegraphics[width=13cm]{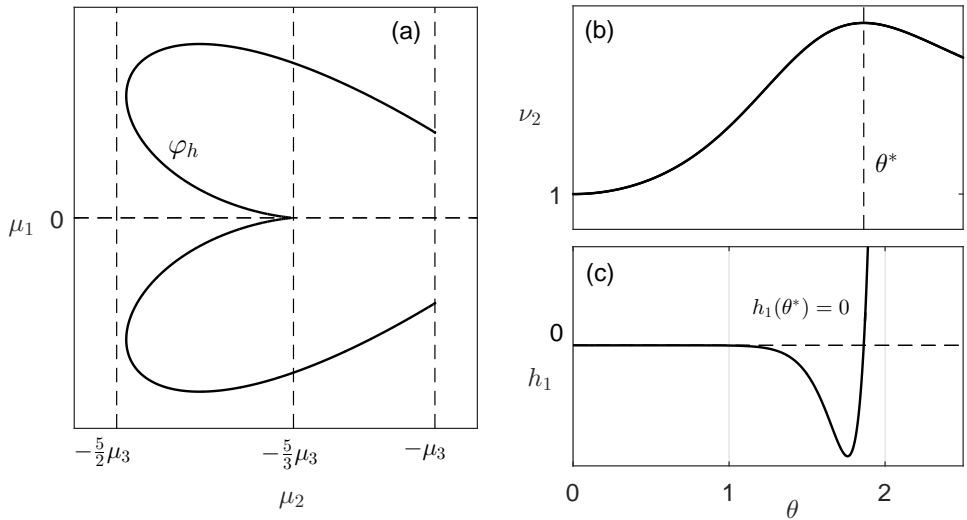}
\end{center}
\caption{(a) The parametric curve defined in \eqref{curve-h} on the plane of
parameters $(\mu_{2},\mu_{1})$. (b) The function $\nu_{2}(\theta)$ defined in
\eqref{parametric-curve}. (c) The function $h(\theta)$ defined in
\eqref{fun-h}.}%
\label{estudio-homoclina}%
\end{figure}

Just to illustrate the quality of the above analytical predictions for the
homoclinic connection bifurcation curve, by using the shooting method (see for
instance \cite{rodriguez1990}) and taking $\mu_{3}=0.1$, we show in Figure
\ref{comparacion}, the numerical continuation curve for the homoclinic orbit
of system \eqref{forma-canonica-1A}, and in red the analytic approximation
curve given by \eqref{curve-h}. As observed, there is a great similarity
between the two approaches, and in general we can conclude that the above
analytical predictions are really useful is getting a global view of the
actual bifurcation set.

\begin{figure}[pth]
\begin{center}
\includegraphics[width=13cm]{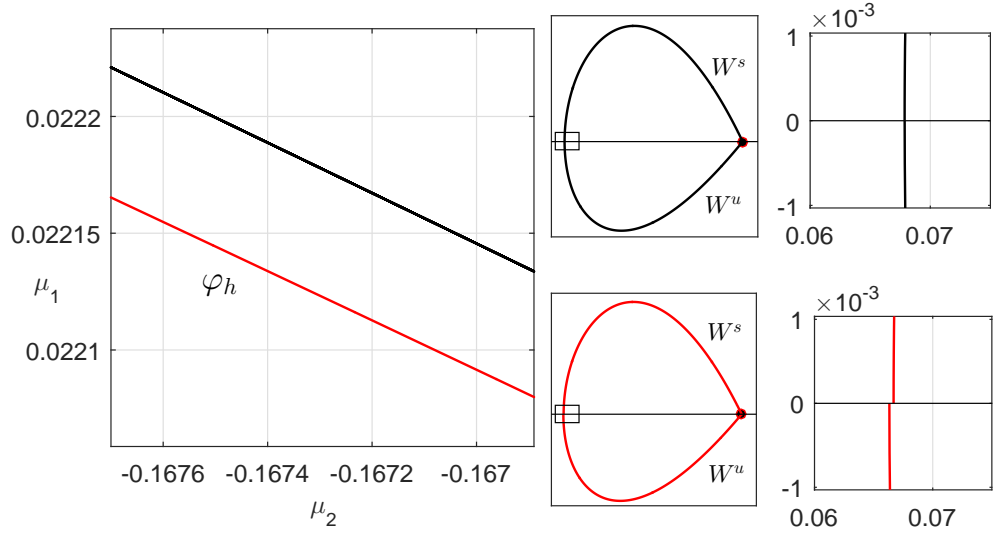}
\end{center}
\caption{System \eqref{forma-canonica-1A} with $\mu_{3}=0.1$. In black (left
panel), the numerical continuation curve in the parameter plane $(\mu_{2}%
,\mu_{1})$, and by using points of the curve, in black (right panel) the
numerical computation of the stable and unstable manifold for a saddle point
of the system . In red (left panel), the analytic approximation curve given by
\eqref{curve-h}, and by using points of the curve, in red (right panel) the
numerical computation of the stable and unstable manifolds for a saddle point
of the system.}%
\label{comparacion}%
\end{figure}


\section{Application to 3D Canonical Memristor Oscillator}

\label{sec:4}

As one of the possible applications of the above study, in this section we will
show the existence of a topological sphere completely foliated by periodic
orbits for a 3D canonical memristor oscillator, when the flux-charge
characteristics of the memristor is a monotone cubic polynomial. The existence
of this sphere was reported numerically in \cite{messias2010,Korneev2017}.

We start by considering the modeling of an elementary oscillator endowed with
one flux-controlled memristor $M$, see Figure \ref{fig:mem3D} and
\cite{chua2008}. In the shown circuit the values of $L$ and $C$ for the
impedance and capacitance are positive constants, while the resistor has a
negative value $-R$. From Kirchoff's laws we see that
\[%
\begin{array}
[c]{rcl}%
i_{R}(\tau)-i_{L}(\tau) & = & 0,\\
i_{L}(\tau)-i_{C}(\tau)-i_{M}(\tau) & = & 0,\\
-v_{R}(\tau)+v_{L}(\tau)+v_{C}(\tau) & = & 0,\\
v_{C}(\tau)-v_{M}(\tau) & = & 0,
\end{array}
\]
where $v,i$ stand for the voltage and current, respectively, across the
corresponding element of the circuit as indicated by the subscript.

\begin{figure}[pth]
\begin{center}
\includegraphics[width=9cm]{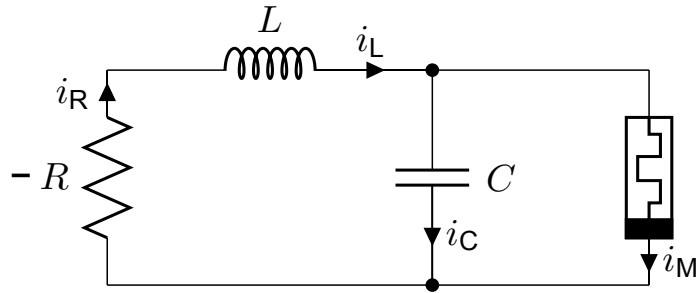}
\end{center}
\caption{The canonical memristor oscillator \cite{chua2008}. Note that the the
only active element in the circuit is the resistor with a negative resistance
$-R$.}%
\label{fig:mem3D}%
\end{figure}In Section 3.2 of \cite{chua2008}, this circuit is proposed as a
third-order canonical memristor oscillator but the notation is slightly
different as follows. They take $i_{1}=i_{C}$, $i_{3}=i_{L}=i_{R}$, $i=i_{M}$,
$v_{1}=v_{C}=v_{M}$, $v_{3}=v_{L}$, $v_{4}=v_{R}$, $\varphi_{1}=\varphi_{C}$,
$\varphi_{3}=\varphi_{L}$, $\varphi_{4}=\varphi_{R}$ and $\varphi=\varphi_{M}%
$. Thus, they write the two equations
\[
i_{1}=i_{3}-i,\quad v_{3}=v_{4}-v_{1},
\]
and, after integrating respect to time, they arrive to
\begin{equation}
q_{1}=q_{3}-q(\varphi),\quad\varphi_{3}=\varphi_{4}-\varphi_{1},\label{eq:2}%
\end{equation}
where $q(\varphi)$ stands for the nonlinear flux-charge characteristics of the
flux-controlled memristor. Solving now for $(q_{3},\varphi_{4})$ and taking
into account that $\varphi_{1}=\varphi$ since $v_{1}=v_{M}$, it is immediate
to obtain
\[
q_{3}=q_{1}+q(\varphi),\quad\varphi_{4}=\varphi+\varphi_{3},
\]
so that authors conclude that a good choice for independent variables is the
triple $(q_{1},\varphi_{3},\varphi)$, that is, the charge of capacitor $C$,
the flux of the inductor $L$ and the flux of the memristor, respectively.
Accordingly, by taking derivatives in \eqref{eq:2}, the following set of
differential equations is proposed,
\begin{equation}%
\begin{array}
[c]{rcl}%
C\dot{v}_{1} & = & i_{3}-W(\varphi)v_{1},\\
L\dot{i}_{3} & = & Ri_{3}-v_{1},\\
\dot{\varphi} & = & v_{1},
\end{array}
\label{eq:3}%
\end{equation}
where $\dot{q}_{1}=i_{1}=C\dot{v}_{1},$ $\dot{q}_{3}=i_{3},$ $\dot{\varphi
}_{4}=v_{4}=Ri_{3}$ and
\[
W(\varphi)=\frac{dq}{d\varphi}.
\]
Finally, they rewrite the system as follows,
\begin{equation}%
\begin{array}
[c]{rcl}%
\dot{x} & = & \alpha\left(  y-W(z)x\right)  ,\\
\dot{y} & = & -\xi x+\beta y,\\
\dot{z} & = & x,
\end{array}
\label{chua2008}%
\end{equation}
where $x=v_{1},\,y=i_{3},\,z=\varphi,$ and the parameters used are
$\alpha=1/C,$ $\xi=1/L,$ and $\beta=R/L,$ so that, $\alpha,\xi,\beta>0$. An
important observation is that the parameter $\alpha$ is not essential so that
it can be removed with the change of variables and parameters
\begin{equation}%
\begin{split}
&  \widetilde{x}=x,\quad\widetilde{y}=\alpha y,\quad\widetilde{z}%
=z,\quad\widetilde{\xi}=\alpha\xi,\\
&  \widetilde{a}=\alpha a,\quad\widetilde{b}=\alpha b,\quad\widetilde
{W}=\alpha W,
\end{split}
\label{removing-alpha}%
\end{equation}
to be assumed in the sequel, omitting also tildes to alleviate the notation.
Therefore, we need to study the system
\begin{equation}%
\begin{array}
[c]{rcl}%
\dot{x} & = & -W(z)x+y,\\
\dot{y} & = & -\xi x+\beta y,\\
\dot{z} & = & x,
\end{array}
\label{ap1:1}%
\end{equation}
where $W(z)=q^{\prime}(z)=3z^{2}+2az+b,$ and
\begin{equation}
q(z)=z^{3}+az^{2}+bz,\label{cubica-ap-1}%
\end{equation}
with $a^{2}-3b<0,$ which assumes that the memristor is passive ($q^{\prime
}(z)>0$).

System \eqref{ap1:1} belongs to a more general class of systems whose
reduction is possible thanks to the existence of a first integral, as shown in
the Appendix. Like other models of memristor oscillators, system \eqref{ap1:1}
has some special feature. For instance, it has a continuum of equilibria on
the z-axis. Furthermore, the Jacobian matrix at any of these points has a zero
eigenvalue.

Taking the parameters $a_{11}=-1,\ a_{12}=1,\ a_{21}=-\xi$ and $a_{22}=\beta$
in Proposition \ref{theor:1} of the appendix, we obtain that for all
$h\in\mathbb{R}$, system \eqref{ap1:1} has an invariant manifold $S_{h}$
defined by
\begin{equation}
S_{h}=\{(x,y,z)\in\mathbb{R}^{3}:-\beta x+y-\beta z^{3}-a\beta z^{2}+\left(
\xi-b\beta\right)  z=h\}. \label{invariant}%
\end{equation}
Moreover, assuming $c=1$ in Corollary \ref{theor-cubic-q} of the appendix, we
obtain that on each invariant manifold $S_{h},$ the system is topologically
equivalent to the Li\'{e}nard system
\begin{equation}%
\begin{split}
\dot{x}  &  =y-x^{3}-x^{2}-(b-\beta)x,\\
\dot{y}  &  =\beta x^{3}+a\beta x^{2}+\left(  b\beta-\xi\right)  x+h.
\end{split}
\label{lienard:ap1}%
\end{equation}
In Figure \ref{invariant-surface}, we show the invariant manifold
\eqref{invariant} corresponding to $h=0.3$ and the set of parameters $\xi
=100$, $a=b=1$, $\beta=5$, along with the phase space of the equivalent
Li\'{e}nard system \eqref{lienard:ap1}. \begin{figure}[pth]
\begin{center}
\includegraphics[width=13cm]{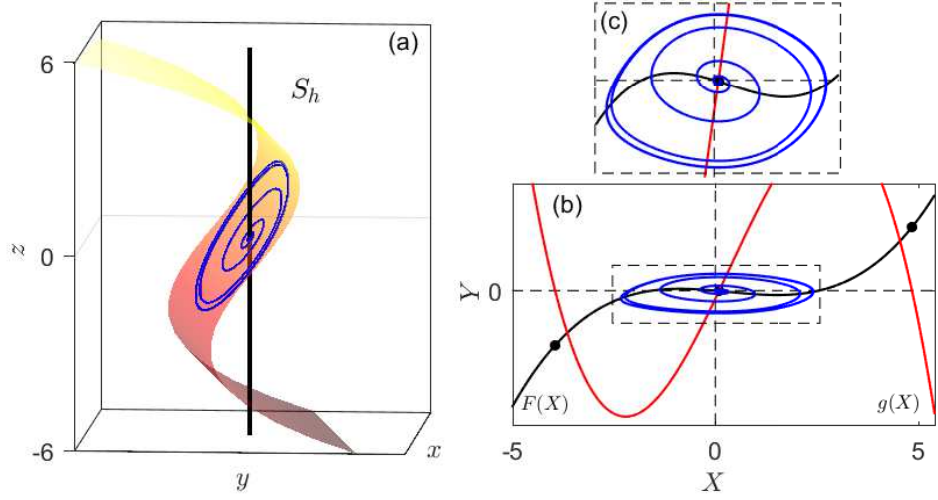}
\end{center}
\caption{(a) The invariant manifold \eqref{invariant} corresponding to the set
of parameters $\xi=100$, $a=b=1$, $\beta=5$ and $h=0.3$ is shown. In black the
infinite number of equilibrium points of the system and in blue a periodic
orbit of the system contained in the invariant manifold. (b) The phase plane
of the equivalent Li\'{e}nard system \eqref{lienard:ap1} corresponding to the
set parameters given in (a) showing a limit cycle in blue, in red the function
$g(X)$ and the function $F(X)$ in black. (c) A zoom of figure (b) is shown.}%
\label{invariant-surface}%
\end{figure}From Proposition \ref{canonical-forms}(a), the system can be
rewritten as
\begin{equation}
\dot{x}=y,\quad\dot{y}=\mu_{1}+\mu_{2}x+\mu_{3}y+x^{3}-3x^{2}y,
\label{norma-form-ap1}%
\end{equation}
where the new parameter are
\begin{equation}%
\begin{split}
\mu_{1}  &  =\frac{1}{27\beta^{5/2}}(27h+9a\xi+2a^{3}\beta-9ab\beta),\\
\mu_{2}  &  =\frac{1}{3\beta^{2}}\left(  \beta(3b-a^{2})-3\xi\right)
,\quad\mu_{3}=\frac{1}{3\beta}\left(  a^{2}-3b+3\beta\right)  .
\end{split}
\label{parameters-ap1}%
\end{equation}
From Remark \ref{existence}, for $\mu_{3}>0$ sufficiently small, we have that
for all $\mu_{2}<-(5/2)\mu_{3}<0$ and $\mu_{1}$ such that
\begin{equation}
|\mu_{1}|<\left(  \frac{\mu_{3}}{3}\right)  ^{3/2}-\left(  \frac{\mu_{3}}%
{3}\right)  ^{3/2}\mu_{2}, \label{cylinder-3}%
\end{equation}
the system has a stable limit cycle. Therefore, we can give the following
result in terms of the parameter $h$, which is associated with the invariant
manifolds $S_{h}$ of 3D system \eqref{sisgeneral}. This result guarantees the
existence of a topological sphere in the 3D phase-space completely foliated by
periodic orbits.

\begin{proposition}
\label{prop-sphare} Consider system \eqref{ap1:1} with $\beta,\xi>0,$, the
function $q$ defined as in \eqref{cubica-ap-1}, $a^{2}-3b<0$ and
\begin{equation}
a^{2}-3b+3\beta>0 \label{cylinder-1}%
\end{equation}
sufficiently small. Additionally, suppose that the following inequalities
hold
\begin{equation}%
\begin{split}
0  &  <3b-a^{2}<3\xi/\beta,\\
\beta(3b-a^{2})-3\xi &  <(5/2)\left(  3b-a^{2}-3\beta\right)  \beta<0.
\end{split}
\label{cylinder-2}%
\end{equation}
Then for all $h\in\mathbb{R}$ with%
\[
-\frac{A}{27}<h<\frac{B}{27},
\]
where
\begin{equation}%
\begin{split}
A  &  =\left(  4a^{2}\beta+3\beta^{2}-12b\beta+9\xi\right)  \sqrt
{a^{2}-3b+3\beta}+9a\xi+2a^{3}\beta-9ab\beta,\\
B  &  =\left(  4a^{2}\beta+3\beta^{2}-12b\beta+9\xi\right)  \sqrt
{a^{2}-3b+3\beta}-9a\xi-2a^{3}\beta+9ab\beta,
\end{split}
\label{limits-AB}%
\end{equation}
the system has a stable periodic orbit. Moreover, there exist a topological
sphere $\Omega$ (see Figure \ref{sphereAp1}) foliated by such periodic orbits. \
\end{proposition}

\begin{proof}
From Remark \ref{existence}, and after substituting the values of $\mu_{1}%
,\mu_{2}$ and $\mu_{3}$ given in \eqref{parameters-ap1}, we obtain the
inequalities \eqref{cylinder-1}-\eqref{cylinder-2}. Now, from
\eqref{cylinder-3} we obtain $|\mu_{1}|<(1/3)\left(  \mu_{3}/3\right)
^{1/2}\left(  \mu_{3}-3\mu_{2}\right)  ,$ so that $\mu_{3}-3\mu_{2}>0,$ since
from hypotheses we have $\mu_{2}<-(5/2)\mu_{3}<0.$ Now after some algebra we obtain%

\[
|27h+9a\xi+2a^{3}\beta-9ab\beta|<\left(  4a^{2}\beta+3\beta^{2}-12b\beta
+9\xi\right)  \sqrt{a^{2}-3b+3\beta}.
\]
Taking into account the absolute value, and grouping terms, we obtain the
values of $A$ and $B$ defined in \eqref{limits-AB}. Finally, from Remark
\ref{existence} system \eqref{ap1:1} has a stable periodic orbit on each
$S_{h}$ defined in \eqref{invariant}, so varying the parameter $h$, we obtain
a sphere foliated by such periodic orbits.
\end{proof}

\begin{figure}[pth]
\begin{center}
\includegraphics[width=12cm]{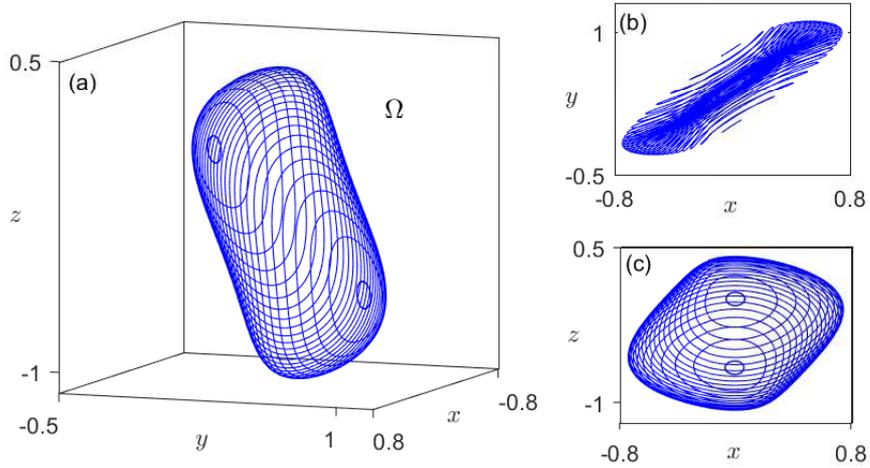}
\end{center}
\caption{Using \eqref{solucion-ap} on each invariant manifold $S_{h}$ defined
in \eqref{invariant}, some slices of the surface $\Omega$ given by Proposition
\ref{prop-sphare} for system \eqref{ap1:1} with parameters $a=1,b=4.8,\beta=5$
and $\xi=80$ are shown. For this set of parameters we get $\mu_{3}=0.106>0$,
$\mu_{2}=-2.3<-(5/2)\mu_{3}$, $A=1180.1$ and $B=152.2$.}%
\label{sphereAp1}%
\end{figure}

\section{False Hidden Attractors in Memristor-Based Autonomous Duffing
Oscillators}

\label{sec:5}

An attractor is called a hidden attractor if its basin of attraction does not
intersect any neighborhood of equilibria; otherwise, it is called a
self-excited attractor, for more details see \cite{Kuznetsov2011,Leonov2011b}.
 Recently in \cite{mentira2018,mentira2018b,mentira2018c} it was
reported the existence of an infinite number of hidden attractors in a
memristor-based autonomous Duffing oscillators, whose  memristance function is
a cubic polynomial. Here, by using a similar approach to the followed in the
previous section, we will show that such hidden attractors are not possible,
so that the numerical simulations included in \cite{mentira2018,mentira2018b,mentira2018c}
 are misleading.

The quoted memristor based autonomous Duffing oscillator is modeled by the
dynamical system
\begin{equation}%
\begin{split}
\dot{x} &  =y,\\
\dot{y} &  =z,\\
\dot{z} &  =-\alpha z-M(x)y,
\end{split}
\label{mentira-1}%
\end{equation}
where the memristance function $M$ (possibly discontinuous) is defined as
\begin{equation}
M(x)=\frac{d\phi(x)}{dx}\label{functionM}%
\end{equation}
and $\phi$ is a continuous function. System \eqref{mentira-1} has a continuum
of equilibria, since any point of the $x$-axis is an equilibrium point. In the
next result,  we show that even system \eqref{mentira-1} does not belong to
the family \eqref{sisgeneral} of the appendix, the system also has the
property of possessing an infinite number of invariant manifolds.

\begin{proposition}
\label{propo:mentira0}Consider system \eqref{mentira-1} with the function $M$
defined as in \eqref{functionM}. For any $h\in\mathbb{R}$ the set
\begin{equation}
S_{h}=\{(x,y,z)\in\mathbb{R}^{3}\colon H(x,y,z)=h\}\label{sh-cubica}%
\end{equation}
is an invariant manifold for the system, where we have introduced the
continuous function
\begin{equation}
H(x,y,z)=\phi(x)+\alpha y+z.\label{invariante-cubica}%
\end{equation}
Therefore, the system has an infinite number of invariant manifolds foliating
the whole $\mathbb{R}^{3}$, and so the dynamics is essentially two-dimensional.
\end{proposition}

\begin{proof}
Taking $H$ as in \eqref{invariante-cubica}, define for any solution
$(x(\tau),y(\tau),z(\tau))$ of \eqref{mentira-1} the auxiliary continuous
function
\[
h(\tau)=H(x(\tau),y(\tau),z(\tau))
\]
Now, a direct computation gives, excepting the points of possible
non-differentiability,
\[
h^{\prime}(\tau)=\frac{d\phi(x)}{dx}\dot{x}+\alpha\dot{y}+\dot{z}=M(x)y+\alpha
z-\alpha z-M(x)y=0.
\]
Then $h$ is piecewise constant along the orbits of \eqref{mentira-1}, but as
$h$ is continuous by definition, it should be globally constant. In short, the
level sets of $H$ are invariant for the flow.
\end{proof}

Now, by using the above result, we reduce the study of the dynamical behavior
of the system, to the study of a planar system.

\begin{proposition}
\label{propo:mentira1}Consider system \eqref{mentira-1} with the function $M$
defined as in \eqref{functionM}. Then on each invariant set $S_{h}$ defined in
\eqref{sh-cubica} the system is topologically equivalent to the planar system
\begin{equation}%
\begin{split}
\dot{x} &  =y,\\
\dot{y} &  =-\phi(x)-\alpha y+h.
\end{split}
\label{reducido-mentira}%
\end{equation}
Moreover, $(x\left(  \tau\right)  ,y\left(  \tau\right)  )\in\mathbb{R}^{2}$
is a solution of the above system if and only if $E_{h}\left(  x\left(
\tau\right)  ,y\left(  \tau\right)  \right)  $ is a solution of system
\eqref{mentira-1}, where
\begin{equation}
E_{h}\left(  X\left(  \tau\right)  ,Y\left(  \tau\right)  \right)  =%
\begin{pmatrix}
x\left(  \tau\right)  \\
y\left(  \tau\right)  \\
h-\phi(x(\tau))-\alpha y(\tau)
\end{pmatrix}
\end{equation}

\end{proposition}

\begin{proof}
From Proposition \ref{propo:mentira0} we can solve for $z$ in the equation
$H(x,y,z)=h$, and write
\[
z=h-\phi(x)-\alpha y.
\]
Replacing this expression into the first and second equation of
\eqref{mentira-1} we obtain system \eqref{reducido-mentira}. Suppose that
$(x\left(  \tau\right)  ,y\left(  \tau\right)  )\in\mathbb{R}^{2}$ is a
solution of system \eqref{reducido-mentira}. Taking
\[
z(\tau)=h-\alpha y(\tau)-\phi(x(\tau))
\]
we obtain
\begin{align*}
\dot{z}(\tau)  &  =-\alpha\dot{y}(\tau)-\frac{d\phi(x(\tau))}{dx}\dot{x}%
(\tau)=-\alpha\left(  h-\phi(x(\tau))-\alpha y(\tau)\right)  -M(x(\tau
))y(\tau)=\\
&  =-\alpha\left(  z(\tau)\right)  -M(x(\tau))y(\tau).
\end{align*}
and the proposition follows.
\end{proof}

In the following result, we show that for $\alpha\neq0$, the system does not
have periodic solutions.

\begin{proposition}
\label{propo:mentira}Consider system \eqref{reducido-mentira}. The following
statements hold.

\begin{itemize}
\item[(a)] For $\alpha=0$ the system is Hamiltonian.

\item[(b)] For $\alpha\neq0$ the system does not have periodic solutions.
\end{itemize}
\end{proposition}

\begin{proof}
The divergence of the system is $\Delta=-\alpha.$ \ Then, when $\alpha=0$ the
system corresponds to the Hamiltonian%
\[
H(x,y)=\frac{y^{2}}{2}+\phi^{\prime}(x).
\]
For $\alpha\neq0$ the divergence of system \eqref{reducido-mentira} does not
change sign, thus from Bendixson's criterion \cite{Kocak91}, system
\eqref{mentira-1} does not have periodic solutions.
\end{proof}

\begin{remark}
\label{mentira}Note that as a consequence of proposition \ref{propo:mentira1}
and \ref{propo:mentira}, the $3D$ system \eqref{mentira-1} system \ cannot
have periodic orbits for any continuous function $\phi$ and $\alpha\neq0.$
However, when $\alpha=0$ the system could have an infinite number of periodic
orbits on each invariant set $S_{h}$ defined in \eqref{sh-cubica}.
\end{remark}

Regarding \cite{mentira2018,mentira2018b,mentira2018c} authors consider system
\eqref{mentira-1} with the function $\phi(x)=\omega x+\beta x^{3}$ and the set
of parameters $\alpha=0.0001,$ $\omega=0.35,$ $\beta=0.85.$ In both quoted
references, authors reported the existence of an infinite number of stable
periodic orbits coexisting with an infinite number of stable equilibria, by
taking into account several numerical simulations, so concluding the existence
of hidden attractors.

From Propositions \ref{propo:mentira0} and \ref{propo:mentira1}, we
obtain\ the invariant manifolds
\[
S_{h}=\{(x,y,z)\in\mathbb{R}^{3}\colon\omega x+\beta x^{3}+\alpha y+z=h\},
\]
and the planar system ruling the dynamics on each $S_{h}$ given by
\[
\dot{x}=y,\quad\dot{y}=-\omega x-\beta x^{3}-\alpha y+h.
\]

From Remark \ref{mentira} we note that, the system cannot have
periodic orbits, and so, the statement made in the quoted papers is clearly
wrong, probably after giving too much credit to numerical simulations. \ This
emphasizes the relevance of the approach followed in this work which allows to
avoid misconceptions coming just from numerical simulations.

\section{Conclusions}

Motivated by the dynamical analysis of 3D memristor oscillators whose
nonlinear characteristics is a cubic polynomial, and after showing that their
dynamics is essentially two-dimensional, the need to consider a disregarded
unfolding of the Bogdanov-Takens singularity naturally arose. The
corresponding bifurcation set, including both local and global bifurcations
has been described. While local bifurcations can be easily detected, the
characterization of global bifurcations parameters curves is  much more
involved; only by resorting to Melnikov's theory it was possible to obtain such
curves providing a complete description of the bifurcation set.

Regarding the considered 3D memristor oscillators, and by working within some
parameters regions of the above bifurcation set, it has been possible to show rigorously
the existence of multiple periodic orbits leading to a topological sphere.

When the same approach is
applied to a different family of 3D memristor oscillators, it has been shown that the oscillations
are not possible, contrarily to what had been recently claimed.

\section*{Acknowledgements}

The first author is supported by Pontificia Universidad Javeriana Cali-Colombia. E. Freire and E. Ponce are partially
supported by MINECO/FEDER grant MTM2015-65608-P and by the
\textit{Consejer\'{\i}a de Econom\'{\i}a, Innovaci\'on, Ciencia y Empleo de la
Junta de Andaluc\'{\i}a} under grant P12-FQM-1658.

\appendix{\bf Appendix: Dimensional reduction in 3D memristor oscillators}

We consider a family of three-dimensional systems, which is general enough to
capture all the mathematical models of memristor oscillators given in
\eqref{ap1:1}. Such family was been studied in \cite{amador2017} and
\cite{ponce2017}, where the authors showed that the dynamics of such a family
of three-dimensional systems is essentially ruled by a one parameter set of
two-dimensional systems. We consider the system
\begin{equation}%
\begin{array}
[c]{rcl}%
\dot{x} & = & a_{11}W(z)x+a_{12}y,\\
\dot{y} & = & a_{21}x+a_{22}y,\\
\dot{z} & = & x,
\end{array}
\label{sisgeneral}%
\end{equation}
where the constants $a_{11},a_{12},a_{21},a_{22}\in\mathbb{R}$ and the
function $W$ allows to define a continuous function
\begin{equation}
q(z)=\int_{0}^{z}W(s)ds.\label{wq}%
\end{equation}

The next result guarantees that the dynamics of system \eqref{sisgeneral} is
essentially two-dimensional, see \cite{amador2017} for a proof.

\begin{proposition}
\label{theor:1}Consider system \eqref{sisgeneral} where the functions $W$ and
$q$ are related as in \eqref{wq}. For any $h\in\mathbb{R}$, the set
\begin{equation}
S_{h}=\{(x,y,z)\in\mathbb{R}^{3}:-a_{22}x+a_{12}y-a_{12}a_{21}z+a_{11}%
a_{22}q(z)=h\}\label{Sh-general}%
\end{equation}
is an invariant manifold for the system. Therefore, the system has an infinite
family of invariant manifolds foliating the whole $\mathbb{R}^{3}$, and so the
dynamics is essentially two-dimensional.
\end{proposition}

In the following result we show that on each invariant set $S_{h}$ given in
\eqref{Sh-general}, and for any continuous function $q$ defined as in
\eqref{wq}, the dynamics is topologically equivalent to a Li\'{e}nard system.
Furthermore, we give for any solution of the Li\'{e}nard system with a given
value of $h$, the corresponding solution of the 3D canonical model
\eqref{sisgeneral}. This result is a generalization of Theorem 3 given in
\cite{amador2017}, where the function $q$ was considered to be a continuous
piecewise linear function.

\begin{proposition}
\label{theor2-cubic}Consider system \eqref{sisgeneral} with the function $q$
defined as in \eqref{wq}. If $a_{12}\neq0$, then on each invariant set $S_{h}$
given in \eqref{Sh-general}, the dynamics is topologically equivalent to the
Li\'{e}nard system \
\begin{equation}
\dot{X}=Y-F(X),\quad\dot{Y}=-g(X)+h,\label{lienard-general}%
\end{equation}
where $F$ and $g$ are given by%
\begin{equation}
F(X)=-a_{11}q(X)-a_{22}X,\quad g(X)=a_{11}a_{22}q(X)-a_{12}a_{21}%
X\label{funcionesFG-general}%
\end{equation}
Moreover, $(X\left(  \tau\right)  ,Y\left(  \tau\right)  )\in\mathbb{R}^{2}$
is a solution of the Li\'{e}nard system \eqref{lienard-general} for a given
$h\in\mathbb{R}$, if and only if $E_{h}\left(  X\left(  \tau\right)  ,Y\left(
\tau\right)  \right)  $ $\in\mathbb{R}^{3}$\ is a solution of system
\eqref{sisgeneral} on $S_{h},$ where
\begin{equation}
E_{h}\left(  X\left(  \tau\right)  ,Y\left(  \tau\right)  \right)  =%
\begin{pmatrix}
Y(\tau)-F(X(\tau))\\
\frac{1}{a_{12}}\left[  (a_{22}^{2}+a_{12}a_{21})Y\left(  \tau\right)
-a_{22}Y\left(  \tau\right)  +h\right]  \\
X\left(  \tau\right)
\end{pmatrix}
.\label{condiciones-iniciales-general}%
\end{equation}

\end{proposition}

\begin{proof}
First, with $a_{12}\neq0$ the change of variables
\begin{equation}
\overline{x}=x,\quad\overline{y}=a_{22}x-a_{12}y,\quad\overline{z}%
=z\label{cambio-nuevo}%
\end{equation}
\ transforms system \eqref{sisgeneral} into the system
\begin{align}
\dot{\overline{x}} &  =f_{1}\left(  \overline{z}\right)  \overline
{x}-\overline{y},\label{lienard1B}\\
\dot{\overline{y}} &  =f_{2}\left(  \overline{z}\right)  \overline
{x},\nonumber\\
\dot{\overline{z}} &  =\overline{x},\nonumber
\end{align}
where the functions $f_{1}$ and $f_{2}$ are defined as%
\begin{equation}
f_{1}(\overline{z})=a_{11}W(\overline{z})+a_{22},\quad\quad f_{2}\left(
\overline{z}\right)  =a_{22}a_{11}W(\overline{z})-a_{12}a_{21}%
.\label{f1f2continuas}%
\end{equation}
From Proposition \ref{theor:1}, the invariant manifolds \eqref{Sh-general} for
system \eqref{lienard1B}-\eqref{f1f2continuas} can be written in the new
variables as
\begin{equation}
\widetilde{S}_{h}=\{(\overline{x},\overline{y},\overline{z})\in\mathbb{R}%
^{3}:-\overline{y}+g(\overline{z})=h\}.\label{sh2general}%
\end{equation}
Now, replacing the condition given in \eqref{sh2general} in the first equation
of \eqref{lienard1B} and removing the unnecessary second equation, we obtain
the system
\begin{equation}%
\begin{array}
[c]{l}%
\dot{\overline{x}}=f_{1}\left(  \overline{z}\right)  \overline{x}%
-g(\overline{z})+h,\\
\dot{\overline{z}}=\overline{x}.
\end{array}
\label{lienard2B}%
\end{equation}
where the function $g$ is defined by
\begin{equation}
g(u)=a_{11}a_{22}q(u)-a_{12}a_{21}u.\label{funcionesG-general}%
\end{equation}
After the change of variables
\begin{equation}%
\begin{array}
[c]{l}%
X=\overline{z},\\
Y=-\tilde{F}(\overline{z})+\overline{x},
\end{array}
\label{cambio2B}%
\end{equation}
where $F$ is
\begin{equation}
\tilde{F}(z)=a_{11}q(z)+a_{22}z,\label{funcionesF-general}%
\end{equation}
we obtain
\[
\dot{X}=\dot{\overline{z}}=\overline{x}=Y+\tilde{F}(X)=Y-(-\tilde{F}(X)),
\]
so that
\begin{align*}
\dot{Y} &  =-\tilde{F}^{\prime}(\overline{z})\dot{\overline{z}}+\dot
{\overline{x}}=-\left(  a_{11}q^{\prime}(\overline{z})+a_{22}\right)  +\left(
f_{1}\left(  \overline{z}\right)  \overline{x}-g(\overline{z})+h\right)  =\\
&  =-f_{1}\left(  \overline{z}\right)  \overline{x}+f_{1}\left(  \overline
{z}\right)  \overline{x}-g(\overline{z})+h=-g(\overline{z})+h,
\end{align*}
and taking $F(X)=-\tilde{F}(X)$ we obtain system \eqref{lienard-general}-\eqref{funcionesFG-general}.

If $\left(  X\left(  \tau\right)  ,Y\left(  \tau\right)  \right)
\in\mathbb{R}^{2}$ is a solution of system
\eqref{lienard-general}-\eqref{funcionesFG-general} for a given $h\in
\mathbb{R}$, we have from \eqref{cambio2B} that
\[%
\begin{pmatrix}
\overline{x}(\tau)\\
\overline{z}(\tau)
\end{pmatrix}
=%
\begin{pmatrix}
Y(\tau)-F(\overline{z}(\tau))\\
X(\tau)
\end{pmatrix}
\]
is a solution of system \eqref{lienard2B}. From \eqref{sh2general}, we obtain
on $\widetilde{S}_{h}$ that $\overline{y}=g(\overline{z})-h$, with $g$ as in
\eqref{funcionesG-general}. Thus,
\[%
\begin{pmatrix}
\overline{x}(\tau)\\
\overline{y}(\tau)\\
\overline{z}(\tau)
\end{pmatrix}
=%
\begin{pmatrix}
Y(\tau)-F(X(\tau))\\
g\left(  X(\tau)\right)  -h\\
X(\tau)
\end{pmatrix}
,
\]
is a solution of system \eqref{lienard1B} on $S_{h}$. Finally, from
\eqref{cambio-nuevo} we obtain for system \eqref{sisgeneral} the solution
$x(\tau)=\overline{x}(\tau)$,
\begin{align*}
y(\tau)  &  =\dfrac{1}{a_{12}}\left[  a_{22}\overline{x}(\tau)-\overline
{y}(\tau)\right]  =\dfrac{1}{a_{12}}\left[  a_{22}Y(\tau)-a_{22}%
F(X(\tau))-g\left(  X(\tau)\right)  +h\right] \\
&  =\dfrac{1}{a_{12}}\left[  a_{22}Y(\tau)+a_{22}\tilde{F}(X(\tau))-g\left(
X(\tau)\right)  +h\right]  ,
\end{align*}
and $z(\tau)=\overline{z}(\tau)$. The conclusion follows from the fact that
for all $X$ we have
\[
a_{22}\tilde{F}(X)-g(X)=(a_{22}^{2}+a_{12}a_{21})X.
\]

\end{proof}

In order to apply the analysis performed to system \eqref{forma-canonica-1A},
in what follows we consider the function $q$ defined by a cubic polynomial,
that is, we assume
\begin{equation}
W(z)=3cz^{2}+2az+b,\quad q(z)=cz^{3}+az^{2}+bz,\label{cubica}%
\end{equation}
with $c\neq0$. As a direct consequence of Propositions \ref{theor:1} and
\ref{theor2-cubic}, we obtain the next result.

\begin{corollary}
\label{theor-cubic-q}Consider system \eqref{sisgeneral} with the functions $q$
and $W$ defined as in \eqref{cubica}. If $a_{12}\neq0$, then on each invariant
set $S_{h}$ given by%
\[
S_{h}=\{(x,y,z)\in\mathbb{R}^{3}:-a_{22}x+a_{12}y+a_{11}a_{22}cz^{3}%
+aa_{11}a_{22}z^{2}+(ba_{11}a_{22}-a_{12}a_{21})z=h\}
\]
the dynamics is topologically equivalent to the Li\'{e}nard system
\begin{equation}%
\begin{split}
\dot{x} &  =y+ca_{11}x^{3}+aa_{11}x^{2}+\left(  ba_{11}+a_{22}\right)  x,\\
\dot{y} &  =-a_{11}a_{22}cx^{3}-a_{11}a_{22}ax^{2}+(a_{12}a_{21}-a_{11}%
a_{22}b)x+h.
\end{split}
\label{lienard-curbica1}%
\end{equation}
Moreover, $(x\left(  \tau\right)  ,y\left(  \tau\right)  )\in\mathbb{R}^{2}$
is a solution of the Li\'{e}nard system \eqref{lienard-curbica1} for a given
$h\in\mathbb{R}$, if and only if $E_{h}\left(  x\left(  \tau\right)  ,y\left(
\tau\right)  \right)  $ $\in\mathbb{R}^{3}$\ is a solution of system
\eqref{sisgeneral} on $S_{h},$ where
\begin{equation}
E_{h}\left(  x\left(  \tau\right)  ,y\left(  \tau\right)  \right)  =%
\begin{pmatrix}
y(\tau)+ca_{11}^{3}x(\tau)^{3}+aa_{11}^{2}x(\tau)^{2}+\left(  ba_{11}%
+a_{22}\right)  x(\tau)^{2}\\
\frac{1}{a_{12}}\left[  (a_{22}^{2}+a_{12}a_{21})y\left(  \tau\right)
-a_{22}y\left(  \tau\right)  +h\right]  \\
x\left(  \tau\right)
\end{pmatrix}
.\label{solucion-ap}%
\end{equation}

\end{corollary}

In the next Proposition, we show that system \eqref{lienard-curbica1} can be
written into the form \eqref{unfolding}.

\begin{proposition}
\label{canonical-forms} The following statements hold for system \eqref{lienard-curbica1}.

\begin{itemize}
\item[(a)] If $a_{22}\neq0$ and \ $a_{11}a_{22}<0$ then the system can be
written into the form%
\begin{equation}
\dot{x}=y,\quad\dot{y}=\mu_{1}+\mu_{2}x+cx^{3}+\mu_{3}y+3ca_{11}x^{2}y.
\label{forma-canonica-1}%
\end{equation}
where the new parameters $\mu_{1},\mu_{2}$ and $\mu_{3}$ are given by
\begin{equation}
\label{parameters-canonica-1}%
\begin{split}
\mu_{1}  &  =\frac{27ch+a_{11}a_{22}a(9cb-2a^{2})-9caa_{12}a_{21}}%
{27c^{2}\left(  -a_{11}a_{22}\right)  ^{5/2}},\\
\mu_{2}  &  =\frac{a_{11}a_{22}(a^{2}-3cb)+3ca_{12}a_{21}}{3c\left(
a_{11}a_{22}\right)  ^{2}},\quad\mu_{3}=\frac{a_{11}(a^{2}-3cb)-3ca_{22}%
}{3ca_{11}a_{22}}.
\end{split}
\end{equation}

\item[(b)] If $a_{22}=0$ then the system can be written into the form%
\begin{equation}
\dot{x}=y,\quad\dot{y}=\mu_{1}+\mu_{2}x+\mu_{3}y+3ca_{11}x^{2}y,
\label{forma-canonica-2}%
\end{equation}
where the new parameters $\mu_{1},\mu_{2}$ and $\mu_{3}$ are defined by%
\begin{equation}
\mu_{1}=h-\frac{aa_{12}a_{21}}{3c},\quad\mu_{2}=a_{12}a_{21},\quad\mu
_{3}=ba_{11}-\frac{a^{2}a_{11}}{3c}. \label{parameters-canonica-2}%
\end{equation}

\end{itemize}
\end{proposition}

\begin{proof}
First, the change of variables
\[
u=x+\frac{a}{3c},\quad v=y+\frac{2}{27}\frac{a^{3}}{c^{2}}a_{11}-\frac{1}%
{3}\frac{a}{c}a_{22}-\frac{1}{3}a\frac{b}{c}a_{11},
\]
transforms system \eqref{lienard-curbica1} into
\begin{equation}%
\begin{split}
\dot{u} &  =v+ca_{11}u^{3}+\lambda_{1}u,\\
\dot{v} &  =-ca_{11}a_{22}u^{3}+\lambda_{2}u+\lambda_{3},
\end{split}
\label{sis-auxiliar-A}%
\end{equation}
where the new parameters are
\begin{equation}%
\begin{split}
\lambda_{1} &  =a_{22}+ba_{11}-\frac{1}{3}\frac{a^{2}}{c}a_{11},\quad
\lambda_{2}=a_{12}a_{21}-ba_{11}a_{22}+\frac{1}{3}\frac{a^{2}}{c}a_{11}%
a_{22},\\
\lambda_{3} &  =h+\frac{1}{3}a\frac{b}{c}a_{11}a_{22}-\frac{1}{3}\frac{a}%
{c}a_{12}a_{21}-\frac{2}{27}\frac{a^{3}}{c^{2}}a_{11}a_{22}.
\end{split}
\label{parameter-auxiliar}%
\end{equation}
If $a_{11}a_{22}<0,$ the change of variable
\[
x=\frac{1}{\left(  -a_{11}a_{22}\right)  ^{1/2}}u,\quad y=v,\quad\tau=\frac
{1}{-a_{11}a_{22}}t,
\]
transforms system \eqref{sis-auxiliar-A}-\eqref{parameter-auxiliar} \ into
\begin{align*}
\dot{x} &  =\frac{1}{\left(  -a_{11}a_{22}\right)  ^{3/2}}y+ca_{11}x^{3}%
-\frac{\lambda_{1}}{a_{11}a_{22}}x,\\
\dot{y} &  =\frac{\lambda_{3}}{-a_{11}a_{22}}+\frac{\lambda_{2}}{\left(
-a_{11}a_{22}\right)  ^{1/2}}x+c\left(  -a_{11}a_{22}\right)  ^{3/2}x^{3}%
\end{align*}
and taking into account that
\[
\ddot{x}=\frac{1}{\left(  -a_{11}a_{22}\right)  ^{3/2}}\dot{y}+3ca_{11}%
x^{2}\dot{x}-\frac{\lambda_{1}}{a_{11}a_{22}}\dot{x},
\]
and after some algebra, statement (a) follows. \newline If $a_{22}=0,$ then
from system \eqref{sis-auxiliar-A}, we obtain statement (b) after a direct computation.
\end{proof}

\end{document}